\documentclass[a4paper, 11pt]{amsart}

\usepackage[french, english]{babel} 
\usepackage[T1]{fontenc}
\usepackage[ansinew]{inputenc}

\title[Fractional Schr\"odinger equations on the circle]{A smoothing effect for the fractional Schr\"odinger equations on the circle and observability}

\date{}

\author{Paul Alphonse}
\address{(Paul Alphonse) ENS de Lyon, UMPA, UMR CNRS 5669, Lyon, France}
\email{paul.alphonse@ens-lyon.fr}

\author{Nikolay Tzvetkov}
\address{(Nikolay Tzvetkov) ENS de Lyon, UMPA, UMR CNRS 5669, Lyon, France}
\email{nikolay.tzvetkov@ens-lyon.fr}


\makeatletter 
	\@namedef{subjclassname@2020}{\textup{2020} Mathematics Subject Classification}
\makeatother


\usepackage[top=3cm, bottom=2cm, left=3cm, right=3cm]{geometry}		

\usepackage{enumitem}

\usepackage{bbm}

\usepackage{array}						

\usepackage{amsfonts}
\usepackage{mathtools}
\usepackage{amsthm}
\usepackage{amssymb}

\usepackage{stmaryrd}

\usepackage{comment}

\usepackage{bbm}

\usepackage[dvipsnames]{xcolor}

\numberwithin{equation}{section}

\usepackage{hyperref}

\hypersetup{	
colorlinks=true,
breaklinks=true,
urlcolor= RedViolet,
linkcolor= red,
citecolor=Blue
}

\usepackage[scr]{rsfso}

\usepackage{comment}

\newtheorem{thm}{Theorem}[section]
\newtheorem{prop}[thm]{Proposition}
\newtheorem{lem}[thm]{Lemma}

\theoremstyle{definition}

\newtheorem{rk}[thm]{Remark}

\DeclareMathOperator{\supp}{supp}

\DeclareMathOperator{\vect}{vect}

\begin{document}

\sloppy

\begin{abstract} We show that, after a  renormalisation, one can define the square of the modulus of the solution of the fractional  Schr\"odinger equations on the circle with data in Sobolev spaces of arbitrary negative index. As an application, we obtain observability estimates with rough controls.
\end{abstract}

\maketitle

\selectlanguage{english}

\section{Introduction and main results}
We consider the fractional Schr\"odinger equation on the circle
\begin{equation}\label{eq:schro}\tag{$S_{\alpha}$}
	i\partial_tu + \vert D_x\vert^{\alpha}u = 0,\quad (t,x)\in\mathbb R\times\mathbb T,
\end{equation}
where $\mathbb T = \mathbb R/2\pi\mathbb Z$ and $\alpha>1$ is a positive real number.
In this equation, the operator $\vert D_x\vert^{\alpha}$ is the Fourier multiplier defined on  $\mathscr S'(\mathbb T)=(C^\infty(\mathbb T))'$ by
$$\widehat{\vert D_x\vert^{\alpha}f}(n)=|n|^\alpha\widehat f(n),\quad n\in\mathbb Z,\, f\in \mathscr S'(\mathbb T),$$
where 
$\widehat\cdot$ denotes the Fourier transform on $\mathscr S'(\mathbb T)$, given by 
$$\widehat f(n)= (2\pi)^{-1}\langle f, e^{-inx}\rangle_{\mathscr S'(\mathbb T),C^\infty(\mathbb T)},\quad n\in\mathbb Z.$$
When $\alpha$ is an even integer $\vert D_x\vert^{\alpha}$  is simply the differential operator 
$
(-1)^{\alpha/2}\partial_x^\alpha.
$
For every $u_0\in \mathscr S'(\mathbb T)$, the unique solution in $C(\mathbb R,   \mathscr S'(\mathbb T))$ of the equation \eqref{eq:schro} 
with initial datum
$$
u\vert_{t=0}= u_0,
$$
naturally denoted $e^{it\vert D_x\vert^{\alpha}}u_0$, is given by the exponential sum
\begin{equation}\label{parvo}
	(e^{it\vert D_x\vert^{\alpha}}u_0)(x) = \sum_{n\in \mathbb Z}\widehat{u_0}(n)\,e^{it|n|^\alpha}e^{inx}.
\end{equation}
Our goal in this work is to describe a remarkable smoothing property enjoyed by \eqref{parvo} and its application to observability estimates. 
\subsection{Smoothing effect} For every $u_0\in \mathscr S'(\mathbb T)$ and every $N\geq0$, we denote by $\Pi_Nu_0$ the truncation of $u_0$ at the energy level $N$, that is,
\begin{equation}\label{20022024E1}
	(\Pi_Nu_0)(x) = \sum_{\vert n\vert\le N}\widehat{u_0}(n)\,e^{inx}.
\end{equation}
Of course $\Pi_Nu_0\in C^\infty(\mathbb T)$. 
For $\sigma\in \mathbb R$, we define the operator $\langle D_x\rangle^{\sigma}$ on  $\mathscr S'(\mathbb T)$ as
$$\widehat{\langle D_x\rangle^{\sigma}f}(n)=\langle n \rangle^\sigma\widehat{f}(n),\quad n\in\mathbb Z,\, f\in \mathscr S'(\mathbb T),$$
where $\langle n \rangle=(1+n^2)^{1/2}$. We can therefore define the Sobolev spaces $H^\sigma(\mathbb T)$ via the norm
$$
\| f\|_{H^\sigma(\mathbb T)}=\| \langle D_x\rangle^{\sigma} f\|_{L^2(\mathbb T)}\,.
$$
Let $u_0 \in L^2(\mathbb T)$. Then, ignoring the time oscillations, it follows directly from \eqref{parvo} that 
$$
\langle D_x\rangle^{\sigma}(e^{it\vert D_x\vert^{\alpha}}u_0)\in C(\mathbb R,H^{-\sigma}(\mathbb T)),
$$
and that 
\begin{equation}\label{first_convergence}
\langle D_x\rangle^{\sigma}(e^{it\vert D_x\vert^{\alpha}}u_0)
\, = \lim_{N\rightarrow+\infty}
\langle D_x\rangle^{\sigma}(e^{it\vert D_x\vert^{\alpha}}\Pi_Nu_0) \quad\text{in $L^\infty (\mathbb R, H^{-\sigma}(\mathbb T))$}.
\end{equation}
Suppose that $\sigma>0$. 
Then, the convergence \eqref{first_convergence} is too weak to imply a convergence of nonlinear expressions of 
$
\langle D_x\rangle^{\sigma}(e^{it\vert D_x\vert^{\alpha}}\Pi_Nu_0).
$
For instance, there is no hope to make converge the sequence of $C^\infty (\mathbb R\times\mathbb T)$ functions
\begin{equation}\label{diverge}
	\big\vert\langle D_x\rangle^{\sigma}(e^{it\vert D_x\vert^{\alpha}}\Pi_N u_0)\big\vert^2,
\end{equation}
simply because their zero Fourier coefficients (in $x$)  equal the numbers 
$$
	\Vert\Pi_Nu_0\Vert^2_{H^{\sigma}(\mathbb T)},
$$
which is a sequence of positive numbers which diverges as soon as $u_0 \in L^2(\mathbb T)$ is such that $u_0 \notin H^\sigma(\mathbb T)$. 
Remarkably, this is the only obstruction to the convergence of \eqref{diverge}, as shown by the next statement which is the first main result in this paper.
\begin{thm}\label{thm:se} Let $\alpha>1$ be a positive real number and $u_0\in L^2(\mathbb T)$. For every $\sigma>0$, the following limit exists in $\mathscr S'(\mathbb R\times\mathbb T)$
\begin{equation}\label{13022025E1}
:\hspace{-3pt}\big\vert\langle D_x\rangle^{\sigma}(e^{it\vert D_x\vert^{\alpha}}u_0)\big\vert^2\hspace{-3pt}:
\, = \lim_{N\rightarrow+\infty}\Big(\big\vert\langle D_x\rangle^{\sigma}(e^{it\vert D_x\vert^{\alpha}}\Pi_Nu_0)\big\vert^2 - \Vert\Pi_Nu_0\Vert^2_{H^{\sigma}(\mathbb T)}\Big).
\end{equation}
Moreover,
$$:\hspace{-3pt}\big\vert\langle D_x\rangle^{\sigma}(e^{it\vert D_x\vert^{\alpha}}u_0)\big\vert^2\hspace{-3pt}:\,\in W^{-s_1,\infty}(\mathbb R, H^{-s_2}(\mathbb T)),$$
with
$$s_1 = \frac{2\sigma}{\alpha-1}\quad\text{and}\quad
s_2 = \left\{\begin{array}{cl}
	2\sigma & \text{when $\sigma>\frac14 - \frac1{4\alpha}$,} \\[5pt]
	\displaystyle\bigg(\frac12-\frac{2\sigma}{\alpha-1}\bigg)_+ & \text{when $\sigma\le\frac14 - \frac1{4\alpha}$,}
\end{array}\right.$$
where $(\frac12-\frac{2\sigma}{\alpha-1})_+$ stands for any positive number $\lambda>0$ such that $\lambda>\frac12-\frac{2\sigma}{\alpha-1}$, and the above convergence actually holds in the space $W^{-s_1,\infty}(\mathbb R, H^{-s_2}(\mathbb T))$.
\end{thm}

\medskip



\medskip

We refer to Remark \ref{10032024R1} where the functional space $W^{-s_1,\infty}(\mathbb R, H^{-s_2}(\mathbb T))$ is defined. 
Moreover, let us add the fact that in the following three situations
$$s_1<\dfrac{2\sigma}{\alpha-1}\quad  \&\quad  s_2\in\mathbb R,\quad   s_1\in\mathbb R\quad  \&\quad  s_2<2\sigma,\quad s_1 = \dfrac{2\sigma}{\alpha-1}\quad  \&\quad  s_2\le\dfrac12 - \dfrac{2\sigma}{\alpha-1},$$
we construct a function $u_0\in L^2(\mathbb T)$ such that the convergence \eqref{13022025E1}
does not hold in the space $W^{-s_1,\infty}(\mathbb R, H^{-s_2}(\mathbb T))$. This implies in particular that when $\sigma>\frac14 - \frac1{4\alpha}$, then the exponents $s_1$ and $s_2$ in Theorem \ref{thm:se} are optimal. However, the situation remains partially unclear when $\sigma\le\frac14 - \frac1{4\alpha}$, since we are not able to perform such a construction when $s_1>\frac{2\sigma}{\alpha-1}$ and $2\sigma\le s_2\le\frac12-\frac{2\sigma}{\alpha-1}$. In this case, we can only conclude to a partial optimality of the exponents $s_1$ and $s_2$.

\medskip

The result of Theorem \ref{thm:se} says that the sequence \eqref{diverge} converges after a renormalization. 
Let us mention that a renormalization of similar spirit is used in the study of nonlinear PDE in the presence of singular randomness (see \cite{Tz}) or in the classical work in quantum field theory (see \cite{Simon}). 
The key difference between  Theorem \ref{thm:se} and \cite{Simon,Tz}  is that in the proof of Theorem \ref{thm:se}, we pass to the limit in a singular nonlinear expression thanks to time oscillation effects while in  \cite{Simon,Tz}  one passes to the limit in a singular nonlinear expression thanks to random oscillation effects. Interestingly, for quadratic expressions in both cases (random or time oscillations), the only obstruction comes from the zero Fourier coefficient. 

\medskip

The result of  Theorem \ref{thm:se}  also says that for every $u_0\in H^s(\mathbb T)$, $s\in\mathbb R$, the sequence 
$
(\partial_x
\vert(e^{it\vert D_x\vert^{\alpha}}\Pi_N u_0)\vert^2)_N
$
converges in $\mathscr S'(\mathbb R\times\mathbb T)$. In other words, for $u_0\in H^s(\mathbb T)$, $s\in\mathbb R$, we may give a sense of the distribution 
$
\partial_x \vert(e^{it\vert D_x\vert^{\alpha}} u_0)\vert^2.
$
\medskip

In Theorem \ref{thm:se} we only deal with a quadratic expression of \eqref{parvo}. It would be interesting to understand how the result extends to higher degree polynomials.
In the case of random oscillations, such an analysis is performed in \cite{Simon}.

\subsection{Observability}
The second main result in this paper is the following one, dealing with the observability properties of the equation \eqref{eq:schro}.

\begin{thm}\label{thm:fracobs} 
Let $\alpha>1$, $T>0$ be a positive time and $b\in L^1(\mathbb T)\setminus\{0\}$ be a non-negative function.
There exists a positive constant $C_{b,T}>0$ such that for every $u_0\in L^2(\mathbb T)$,
\begin{equation}\label{18032024E1}
	\Vert u_0\Vert^2_{L^2(\mathbb T)}\le C_{b,T}\int_0^T\int_{\mathbb T}b(x)\big\vert(e^{it\vert D_x\vert^{\alpha}}u_0)(x)\big\vert^2\,\mathrm dx\,\mathrm dt.
\end{equation}
\end{thm}

\medskip

The observability properties of type \eqref{18032024E1} of Schr\"odinger equations on tori $\mathbb T^d$ have been widely studied, mostly in the non-fractional case $\alpha=2$ when $b=\mathbbm1_{\omega}$ and $\omega\subset\mathbb T^d$ is an open set. Among all this literature, the two articles \cite{BZ, LBM} are exceptions, since they consider more general controls $b\in L^1(\mathbb T)$ as in the present paper. Theorem \ref{thm:fracobs} is in particular a generalization of \cite[Lemma 2.4]{BZ} to the fractional case $\alpha>1$. On the other hand, the paper \cite{BZ} allows to treat the two-dimensional case. Our proof basically follows the strategy of \cite{BZ}. However, the techniques used in order to prove Lemma \ref{lem:weakobs} are different from \cite{BZ}. Indeed, the article \cite{BZ} relies on an approach based on results of propagation of singularities \textit{via} the notion of semiclassical defect measure. The present paper, as for it, combines the smoothing effect stated in Theorem \ref{thm:se} (not in its full strength) with an ergodic argument, which is simplified to a projection argument when $b\in L^2(\mathbb T)$ (including the interesting case where $b=\mathbbm1_{\omega}$ and $\omega\subset\mathbb T$ is a measurable set with positive measure). Besides, we hope that our approach will allow us to deal with dispersive equations on higher dimensional tori $\mathbb T^d$.

\medskip

It should be mentioned that Theorem \ref{thm:fracobs} improves
\cite[Theorem 1]{M} when $d=1$, which states an observability estimate for the equation \eqref{eq:schro} in any positive time $T>0$ with localized functions of the form $b=\mathbbm1_{\omega}$ with $\omega\subset\mathbb T$ an open set. On the other hand, \cite[Theorem 1]{M} also treats the higher-dimensional case. 

\medskip

In addition, let us mention that the assumption $\alpha>1$ in Theorem~\ref{thm:fracobs} cannot be relaxed. Indeed, we know from \cite[Theorem 1]{M} that on the one hand, in the situation where $\alpha = 1$, the observability properties of the equation (\hyperref[eq:schro]{$S_1$}) are linked to the so-called Geometric Control Condition on $\omega\subset\mathbb T$. On the other hand, in the setting $0<\alpha<1$, the very same result \cite[Theorem 1]{M} states that the set $\omega\subset\mathbb T$ has to be dense so that an observability estimate holds.

\medskip

{\bf Acknowledgements.} We are grateful to Sebastian Herr for an interesting discussion on the smoothing properties of fractional Schr\"odinger equations. 
We are also very grateful to the anonymous referee for the valuable remarks which lead to the optimality results discussed after Theorem \ref{thm:se}.
N.T. was partially supported by the ANR project Smooth ANR-22-CE40-0017.

\section{Smoothing estimates}

In this section, we give the proof of Theorem \ref{thm:se}. 

\subsection{Prolegomena}\label{sub:pro}

We therefore consider $\alpha>1$ a positive real number, $\sigma>0$ and $u_0\in L^2(\mathbb T)$. Recall that the projectors $\Pi_N$ are defined in \eqref{20022024E1}. First notice that the conservation of mass property of the equation \eqref{eq:schro} implies
$$\Pi_0\big(\big\vert\langle D_x\rangle^{\sigma}(e^{it\vert D_x\vert^{\alpha}}\Pi_Nu_0)\big\vert^2\big) 
= \big\Vert\langle D_x\rangle^{\sigma}(e^{it\vert D_x\vert^{\alpha}}\Pi_Nu_0)\big\Vert^2_{L^2(\mathbb T)}
= \Vert\Pi_Nu_0\Vert^2_{H^{\sigma}(\mathbb T)}.$$
As a consequence, the elements of the sequence
\begin{equation}\label{22022024E3}
	\Big(\big\vert\langle D_x\rangle^{\sigma}(e^{it\vert D_x\vert^{\alpha}}\Pi_Nu_0)\big\vert^2 - \Vert\Pi_Nu_0\Vert^2_{H^{\sigma}(\mathbb T)}\Big)_{N\geq0},
\end{equation}
can be written in the following way
$$\big\vert\langle D_x\rangle^{\sigma}(e^{it\vert D_x\vert^{\alpha}}\Pi_Nu_0)\big\vert^2 - \Vert\Pi_Nu_0\Vert^2_{H^{\sigma}(\mathbb T)} 
= \Pi^{\perp}_0\big(\big\vert\langle D_x\rangle^{\sigma}(e^{it\vert D_x\vert^{\alpha}}\Pi_Nu_0)\big\vert^2\big).$$
We therefore deduce that for every test function $\varphi\in\mathscr S(\mathbb R\times\mathbb T)$, we have
\begin{align*}
	& \int_{\mathbb R}\int_{\mathbb T}\varphi(t,x)\Big(\big\vert\langle D_x\rangle^{\sigma}(e^{it\vert D_x\vert^{\alpha}}\Pi_Nu_0)(x)\big\vert^2 - \Vert\Pi_Nu_0\Vert^2_{H^{\sigma}(\mathbb T)}\Big)\,\mathrm dx\,\mathrm dt \\[5pt]
	= \ & \int_{\mathbb R}\int_{\mathbb T}\varphi(t,x)\,\Pi^{\perp}_0\big(\big\vert\langle D_x\rangle^{\sigma}(e^{it\vert D_x\vert^{\alpha}}\Pi_Nu_0)\big\vert^2\big)(x)\,\mathrm dx\,\mathrm dt \\[5pt]
	= \ & \int_{\mathbb R}\int_{\mathbb T}(\Pi^{\perp}_0\varphi)(t,x)\big\vert\langle D_x\rangle^{\sigma}(e^{it\vert D_x\vert^{\alpha}}\Pi_Nu_0)(x)\big\vert^2\,\mathrm dx\,\mathrm dt.
\end{align*}
The convergence of the sequence \eqref{22022024E3} in $\mathscr S'(\mathbb R\times\mathbb T)$ and the regularity of the limit $:~\hspace{-7pt}\vert\langle D_x\rangle^{\sigma}(e^{it\vert D_x\vert^{\alpha}}u_0)\vert^2\hspace{-3pt}:$ will therefore be a consequence of the following bilinear smoothing estimates.

\begin{prop}\label{prop:ineqtest} For every $\sigma>0$, there exists a positive constant $c_{\sigma}>0$ such that for every test function $\varphi\in\mathscr S(\mathbb R\times\mathbb T)$ and for every initial data $u_0,v_0\in H^{-\sigma}(\mathbb T)$,
\begin{multline*}
	\bigg\vert\int_{\mathbb R}\int_{\mathbb T}(\Pi^{\perp}_0\varphi)(t,x)(e^{it\vert D_x\vert^{\alpha}}u_0)(x)\overline{(e^{it\vert D_x\vert^{\alpha}}v_0)(x)}\,\mathrm dx\,\mathrm dt\bigg\vert \\
	\le c_{\sigma}\Vert\varphi\Vert_{W^{s_1,1}(\mathbb R, H^{s_2}(\mathbb T))}\Vert u_0\Vert_{H^{-\sigma}(\mathbb T)}\Vert v_0\Vert_{H^{-\sigma}(\mathbb T)},
\end{multline*}
with
$$s_1 = \frac{2\sigma}{\alpha-1}\quad\text{and}\quad
s_2 = \left\{\begin{array}{cl}
	2\sigma & \text{when $\sigma>\frac14 - \frac1{4\alpha}$,} \\[5pt]
	\displaystyle\bigg(\frac12-\frac{2\sigma}{\alpha-1}\bigg)_+ & \text{when $\sigma\le\frac14 - \frac1{4\alpha}$,}
\end{array}\right.$$
where $(\frac12-\frac{2\sigma}{\alpha-1})_+$ stands for any positive number $\lambda>0$ such that $\lambda>\frac12-\frac{2\sigma}{\alpha-1}$.
\end{prop}

\begin{rk}\label{10032024R1} The functional space $W^{s_1,1}(\mathbb R, H^{s_2}(\mathbb T))$ appearing in the above statement is defined by
$$W^{s_1,1}(\mathbb R, H^{s_2}(\mathbb T)) = \big\{\varphi\in\mathscr S'(\mathbb R\times\mathbb T) : \langle D_t\rangle^{s_1}\varphi\in L^1(\mathbb R,H^{s_2}(\mathbb T))\big\},$$
and is endowed with the following norm
$$\Vert\varphi\Vert_{W^{s_1,1}(\mathbb R, H^{s_2}(\mathbb T))} = \int_{\mathbb R}\Vert\langle D_t\rangle^{s_1}\varphi(t,\cdot)\Vert_{H^{s_2}(\mathbb T)}\,\mathrm dt.$$
For every $s\in\mathbb R$, the operator $\langle D_t\rangle^s$ is defined on $\mathscr S'(\mathbb R)$ by
$$\mathscr F_t(\langle D_t\rangle^sf)(\tau)=\langle\tau\rangle^s\mathscr F_t(f)(\tau),\quad \tau\in\mathbb R,\, f\in \mathscr S'(\mathbb R),$$
where $\mathscr F_t$ denotes the Fourier transform with respect to the time variable $t\in\mathbb R$.
Moreover, let us introduce the following space
$$W^{-s_1,\infty}(\mathbb R, H^{-s_2}(\mathbb T)) = \big\{\varphi\in\mathscr S'(\mathbb R\times\mathbb T) : \langle D_t\rangle^{-s_1}\varphi\in L^{\infty}(\mathbb R,H^{-s_2}(\mathbb T))\big\},$$
equipped with the norm
$$\Vert\varphi\Vert_{W^{-s_1,\infty}(\mathbb R, H^{-s_2}(\mathbb T))} = \sup_{t\in\mathbb R}\Vert\langle D_t\rangle^{-s_1}\varphi(t,\cdot)\Vert_{H^{-s_2}(\mathbb T)}.$$
By classical arguments, we have the following duality in norm
\begin{equation}\label{13012025E1}
	\Vert\varphi\Vert_{W^{-s_1,\infty}(\mathbb R, H^{-s_2}(\mathbb T))} 
	= \sup_{\Vert\psi\Vert_{W^{s_1,1}(\mathbb R, H^{s_2}(\mathbb T))}\le1}\bigg\vert\int_{\mathbb R\times\mathbb T}\varphi(t,x)\psi(t,x)\,\mathrm dx\,\mathrm dt\bigg\vert.
\end{equation}
\end{rk}

Before proving Proposition \ref{prop:ineqtest}, let us check that Theorem \ref{thm:se} 
is a consequence of this result. Let $u_0\in L^2(\mathbb T)$ and 
$\varphi\in\mathscr S(\mathbb R\times\mathbb T)$. In order to simplify the writing, 
we set 
$$u_N(t,x) = \langle D_x\rangle^{\sigma}(e^{it\vert D_x\vert^{\alpha}}\Pi_Nu_0)(x),\quad(t,x)\in\mathbb R\times\mathbb T,$$
and also
\begin{align*}
	c_N(u) 
	 = \int_{\mathbb R}\int_{\mathbb T}(\Pi^{\perp}_0\varphi)(t,x)\vert u_N(t,x)\vert^2\,\mathrm dx\,\mathrm dt.
\end{align*}
Recall that we first aim at proving that the sequence $(c_N(u_0))_N$ converges. 
By combining the following decomposition
$$\vert z_1\vert^2 - \vert z_2\vert^2 = z_1\overline{z_1} - z_2\overline{z_2} = z_1(\overline{z_1} - \overline{z_2}) + \overline{z_2}(z_1-z_2),\quad z_1,z_2\in\mathbb C,$$
and the bilinear estimate given by Proposition \ref{prop:ineqtest}, we get that for every 
$N_1,N_2\geq0$,
\begin{align}\label{13012025E2}
	&\ \vert c_{N_1}(u_0) - c_{N_2}(u_0)\vert \\[5pt]
	\lesssim_{\sigma} &\ \Vert\varphi\Vert_{W^{s_1,1}(\mathbb R, H^{s_2}(\mathbb T))}(\Vert u_{N_1}\Vert_{H^{-\sigma}(\mathbb T)} + \Vert u_{N_2}\Vert_{H^{-\sigma}(\mathbb T)})\Vert u_{N_1} - u_{N_2}\Vert_{H^{-\sigma}(\mathbb T)} \nonumber \\[5pt]
	\lesssim_{\sigma} &\ \Vert\varphi\Vert_{W^{s_1,1}(\mathbb R, H^{s_2}(\mathbb T))}\Vert u_0\Vert_{L^2(\mathbb T)}\Vert\Pi_{N_1}u_0 - \Pi_{N_2}u_0\Vert_{L^2(\mathbb T)}
	\underset{N_1,N_2\rightarrow+\infty}{\longrightarrow}0. \nonumber
\end{align}
This proves that the sequence $(c_N(u_0))_N$ is a numerical Cauchy sequence, which therefore converges. As a consequence, the sequence \eqref{22022024E3} converges in $\mathscr S'(\mathbb R\times\mathbb T)$. 
Now, let us denote by $f_N$ the elements of the sequence \eqref{22022024E3}. It follows from the duality relation \eqref{13012025E1} and the inequality \eqref{13012025E2} that
$$\Vert f_{N_1} - f_{N_2}\Vert_{W^{-s_1,\infty}(\mathbb R, H^{-s_2}(\mathbb T))}\lesssim_{\sigma}\Vert u_0\Vert_{L^2(\mathbb T)}\Vert\Pi_{N_1}u_0 - \Pi_{N_2}u_0\Vert_{L^2(\mathbb T)}\underset{N_1,N_2\rightarrow+\infty}{\longrightarrow}0.$$
This estimate proves that the sequence \eqref{22022024E3} is a Cauchy sequence in the Banach space $W^{-s_1,\infty}(\mathbb R, H^{-s_2}(\mathbb T))$, and therefore converges in this space. The proof of Theorem \ref{thm:se} is now ended.

\subsection{Bilinear estimates} We now give the proof of Proposition \ref{prop:ineqtest}. Let us consider $\sigma>0$, $\varphi\in\mathscr S(\mathbb R\times\mathbb T)$, $u_0,v_0\in L^2(\mathbb T)$ be fixed functions, and write
$$u_0(x) = \sum_{n\in\mathbb Z}a_ne^{inx},\quad v_0(x) = \sum_{n\in\mathbb Z}b_ne^{inx},\quad x\in\mathbb T,$$
and
$$(\Pi^{\perp}_0\varphi)(t,x) = \sum_{n\in\mathbb Z\setminus\{0\}}c_n(t)e^{inx},\quad(t,x)\in\mathbb R\times\mathbb T.$$
A straightforward integration in space first shows that
$$\int_{\mathbb T}(\Pi_0^{\perp}\varphi)(t,x)(e^{it\vert D_x\vert^{\alpha}}u_0)(x)\overline{(e^{it\vert D_x\vert^{\alpha}}v_0)(x)}\,\mathrm dx 
= \sum_{\substack{n+ n_1 - n_2 = 0,\\ n\ne0}}c_n(t)a_{n_1}\overline{b_{n_2}}e^{it(\vert n_1\vert^{\alpha} - \vert n_2\vert^{\alpha})},$$
where $\mathrm dx$ is properly normalized. The integration in time then makes the Fourier transform of the coefficients $c_n$ appear as follows
$$\int_{\mathbb R}c_n(t)e^{it(\vert n_1\vert^{\alpha} - \vert n_2\vert^{\alpha})}\,\mathrm dt = \widehat{c_n}(\vert n_2\vert^{\alpha} - \vert n_1\vert^{\alpha}).$$
Thus, the term we aim at estimating is given by
$$\int_{\mathbb R}\int_{\mathbb T}(\Pi_0^{\perp}\varphi)(t,x)(e^{it\vert D_x\vert^{\alpha}}u_0)(x)\overline{(e^{it\vert D_x\vert^{\alpha}}v_0)(x)}\,\mathrm dx
= \sum_{\substack{n+ n_1 - n_2 = 0,\\ n\ne0}}\widehat{c_n}(\vert n_2\vert^{\alpha} - \vert n_1\vert^{\alpha})a_{n_1}\overline{b_{n_2}}.$$
Notice that since $n\ne 0$ in the above sum, we have $n_1\ne n_2$. There are then two cases to consider for the integers $n_1$ and $n_2$. \\[2pt]
\textbf{$\triangleright$ Resonant case.}
In the case where $n_1 = - n_2$, we get $n=-2n_1$ and Cauchy-Schwarz' inequality gives the following bound
\begin{align*}
	\sum_{\substack{n+ n_1 - n_2 = 0,\\ n\ne0,\,n_1=-n_2}}\vert\widehat{c_n}(\vert n_2\vert^{\alpha} - \vert n_1\vert^{\alpha})\vert\vert a_{n_1}\vert\vert b_{n_2}\vert
	& = \sum_{n_1\in\mathbb Z}\vert\widehat{c_{-2n_1}}(0)\vert\vert a_{n_1}\vert\vert b_{-n_1}\vert \\
	& \le\sup_{n_1\in\mathbb Z}\langle n_1\rangle^{2\sigma}\vert\widehat{c_{n_1}}(0)\vert\sum_{n_1\in\mathbb Z}\frac{\vert a_{n_1}\vert\vert b_{-n_1}\vert}{\langle 2n_1\rangle^{2\sigma}} \\[5pt]
	& \lesssim\sup_{n_1\in\mathbb Z}\langle n_1\rangle^{2\sigma}\vert\widehat{c_{n_1}}(0)\vert\Vert u_0\Vert_{H^{-\sigma}(\mathbb T)}\Vert v_0\Vert_{H^{-\sigma}(\mathbb T)}.
\end{align*}
Moreover, the above supremum can be estimated by noticing that for every $n_1\in\mathbb Z$,
\begin{multline*}
	\langle n_1\rangle^{2\sigma}\vert\widehat{c_{n_1}}(0)\vert
	\le\langle n_1\rangle^{2\sigma}\sup_{\tau\in\mathbb R}\,\langle\tau\rangle^s\vert\widehat{c_{n_1}}(\tau)\vert
	\le\int_{\mathbb R}\langle n_1\rangle^{2\sigma}\vert\langle D_t\rangle^sc_n(t)\vert\,\mathrm dt \\
	\le\int_{\mathbb R}\Vert\langle D_t\rangle^s\varphi(t,\cdot)\Vert_{H^{2\sigma}(\mathbb T)}\,\mathrm dt
	= \Vert\varphi\Vert_{W^{s,1}(\mathbb R,H^{2\sigma}(\mathbb T))},
\end{multline*}
where we set $s = 2\sigma/(\alpha-1)>0$. This value for $s>0$ will be justified later, while studying the non-resonant case. We therefore deduce that
\begin{equation}\label{09012025E2}
	\sum_{\substack{n+ n_1 - n_2 = 0,\\ n\ne0,\,n_1=-n_2}}\vert\widehat{c_n}(\vert n_2\vert^{\alpha} - \vert n_1\vert^{\alpha})\vert\vert a_{n_1}\vert\vert b_{n_2}\vert
	\lesssim\Vert\varphi\Vert_{W^{s,1}(\mathbb R,H^{2\sigma}(\mathbb T))}\Vert u_0\Vert_{H^{-\sigma}(\mathbb T)}\Vert v_0\Vert_{H^{-\sigma}(\mathbb T)}.
\end{equation}
\textbf{$\triangleright$ Non-resonant case.}
In the other situation where $n_1\ne-n_2$, we write
\begin{align}\label{07042024E2}
	\sum_{\substack{n+ n_1 - n_2 = 0,\\ n\ne0,\,n_1\ne-n_2}}\vert\widehat{c_n}(\vert n_2\vert^{\alpha} - \vert n_1\vert^{\alpha})\vert\vert a_{n_1}\vert\vert b_{n_2}\vert 
	\le \sum_{\substack{n+ n_1 - n_2 = 0,\\ n\ne0,\,n_1\ne-n_2}}\Gamma_n\,\frac{\vert a_{n_1}\vert\vert b_{n_2}\vert}{\langle\vert n_1\vert^{\alpha}-\vert n_2\vert^{\alpha}\rangle^s}, 
\end{align}
where $s = 2\sigma/(\alpha-1)>0$ is the same as before, and where we set 
\begin{equation}\label{18032024E5}
	\Gamma_n = \Vert\langle\tau\rangle^s\widehat{c_n}(\tau)\Vert_{L^{\infty}(\mathbb R)},\quad n\in\mathbb Z.
\end{equation}
In order to treat the sum in the right-hand side of the estimate \eqref{07042024E2}, we need to establish the following lemma.

\begin{lem}\label{13032024L1} For every real numbers $x,y\in\mathbb R$, we have
$$\vert\vert x\vert^{\alpha} - \vert y\vert^{\alpha}\vert\geq\frac1{2^{\alpha-1}}(\vert x\vert+\vert y\vert)^{\alpha-1}\vert\vert x\vert-\vert y\vert\vert.$$
\end{lem}
\begin{proof}
It sufficient to consider the case where $x,y\geq0$. Since $\alpha>1$, we can write
\begin{equation}\label{24012025E2}
	x^{\alpha} - y^{\alpha} = \int_0^1\frac{\mathrm d}{\mathrm dt}(tx+(1-t)y)^{\alpha}\,\mathrm dt = \int_0^1\alpha(tx+(1-t)y)^{\alpha-1}(x-y)\,\mathrm dt.
\end{equation}
We therefore get that
\begin{align*}
	\vert x^{\alpha} - y^{\alpha}\vert & = \bigg(\int_0^1\alpha(tx+(1-t)y)^{\alpha-1}\,\mathrm dt\bigg)\vert x-y\vert\\[5pt]
	& \geq\bigg(\int_0^1\alpha\min(t,1-t)^{\alpha-1}\,\mathrm dt\bigg)(x+y)^{\alpha-1}\vert x-y\vert \\[5pt]
	& = \frac1{2^{\alpha-1}}(x+y)^{\alpha-1}\vert x-y\vert.
\end{align*}
This is the expected estimate. 
\end{proof}

We then deduce from Lemma \ref{13032024L1} that when $n_1\ne n_2$, we have
$$\vert\vert n_1\vert^{\alpha} - \vert n_2\vert^{\alpha}\vert\gtrsim(\vert n_1\vert + \vert n_2\vert)^{\alpha-1}\vert\vert n_1\vert - \vert n_2\vert\vert,$$
and since $\vert\vert n_1\vert - \vert n_2\vert\vert\geq1$, we get that
\begin{align*}
	\langle\vert n_1\vert^{\alpha} - \vert n_2\vert^{\alpha}\rangle^2
	& \gtrsim\vert\vert n_1\vert^{\alpha} - \vert n_2\vert^{\alpha}\vert^2 \\[5pt]
	& \gtrsim (\vert n_1\vert + \vert n_2\vert)^{2(\alpha-1)}\vert\vert n_1\vert - \vert n_2\vert\vert^2 \\[5pt]
	& \gtrsim (1 + \vert n_1\vert^2 + \vert n_2\vert^2)^{\alpha-1}\langle\vert n_1\vert - \vert n_2\vert\rangle^2 \\[5pt]
	& \gtrsim\langle n_1\rangle^{\alpha-1}\langle n_2\rangle^{\alpha-1}\langle\vert n_1\vert - \vert n_2\vert\rangle^2.
\end{align*}
The parameter $s>0$ is chosen in order to satisfy $s(\alpha-1)/2 = \sigma$. 
We therefore deduce that the sum we are studying is bounded in the following way
$$\sum_{\substack{n+ n_1 - n_2 = 0,\\ n\ne0,\,n_1\ne-n_2}}\vert\widehat{c_n}(\vert n_2\vert^{\alpha} - \vert n_1\vert^{\alpha})\vert\vert a_{n_1}\vert\vert b_{n_2}\vert 
\le\sum_{\substack{n_1,n_2\in\mathbb Z \\ n_1\ne\pm n_2}}\frac{\Gamma_{n_1-n_2}}{\langle\vert n_1\vert - \vert n_2\vert\rangle^s}\,\frac{\vert a_{n_1}\vert}{\langle n_1\rangle^{\sigma}}\frac{\vert b_{n_2}\vert}{\langle n_2\rangle^{\sigma}}.$$
In order to alleviate the writing, let us set
$$K_{n_1,n_2} = \frac{\Gamma_{n_1-n_2}}{\langle\vert n_1\vert - \vert n_2\vert\rangle^s},\quad \tilde a_{n_1} = \frac{\vert a_{n_1}\vert}{\langle n_1\rangle^{\sigma}}\quad\text{and}\quad
\tilde b_{n_2} = \frac{\vert b_{n_2}\vert}{\langle n_2\rangle^{\sigma}}.$$
We deduce from Cauchy-Schwarz' inequality that
\begin{align*}
	\sum_{\substack{n_1,n_2\in\mathbb Z \\ n_1\ne\pm n_2}}K_{n_1,n_2}\tilde a_{n_1}\tilde b_{n_2} 
	& = \sum_{n_2\in\mathbb Z}\bigg(\sum_{\substack{n_1\in\mathbb Z \\ n_1\ne\pm n_2}}K_{n_1,n_2}\tilde a_{n_1}\bigg)\tilde b_{n_2} \\
	& \le\bigg(\sum_{n_2\in\mathbb Z}\bigg(\sum_{\substack{n_1\in\mathbb Z \\ n_1\ne\pm n_2}}K_{n_1,n_2}\tilde a_{n_1}\bigg)^2\bigg)^{1/2}\bigg(\sum_{n_2\in\mathbb Z}\tilde b_{n_2}^2\bigg)^{1/2} \\
	& = \bigg(\sum_{n_2\in\mathbb Z}\bigg(\sum_{\substack{n_1\in\mathbb Z \\ n_1\ne\pm n_2}}K_{n_1,n_2}\tilde a_{n_1}\bigg)^2\bigg)^{1/2}\Vert v_0\Vert_{H^{-\sigma}(\mathbb T)}.
\end{align*}
We need to introduce two new notations still in order to alleviate the writting
$$A_1 = \sup_{n_1\in\mathbb Z}\sum_{\substack{n_2\in\mathbb Z \\ n_2\ne\pm n_1}}K_{n_1,n_2}\quad\text{and}\quad A_2 = \sup_{n_2\in\mathbb Z}\sum_{\substack{n_1\in\mathbb Z \\ n_1\ne\pm n_2}}K_{n_1,n_2}.$$
The above double sum is also controlled with Cauchy-Schwarz' inequality as follows
\begin{align*}
	\sum_{n_2\in\mathbb Z}\bigg(\sum_{\substack{n_1\in\mathbb Z \\ n_1\ne\pm n_2}}K_{n_1,n_2}\tilde a_{n_1}\bigg)^2
	& \le\sum_{n_2\in\mathbb Z}\bigg(\sum_{\substack{n_1\in\mathbb Z \\ n_1\ne\pm n_2}}K_{n_1,n_2}\bigg)\bigg(\sum_{\substack{n_1\in\mathbb Z \\ n_1\ne\pm n_2}}K_{n_1,n_2}\tilde a_{n_1}^2\bigg) \\
	& \le A_2\sum_{n_1\in\mathbb Z}\tilde a_{n_1}^2\bigg(\sum_{\substack{n_2\in\mathbb Z \\ n_2\ne\pm n_1}}K_{n_1,n_2}\bigg) \\
	& \le A_1A_2\Vert u_0\Vert^2_{H^{-\sigma}(\mathbb T)}.
\end{align*}
It remains to study $A_1$ and $A_2$. Since these two terms are similar, we only focus on $A_1$. To that end, we consider a non-negative number $\lambda\geq0$ such that $\lambda + s>1/2$. Let $n_1\in\mathbb Z$ be fixed. We deduce once more from Cauchy-Schwarz' inequality that
\begin{multline*}
	\sum_{\substack{n_2\in\mathbb Z \\ n_2\ne\pm n_1}}K_{n_1,n_2} 
	= \sum_{\substack{n_2\in\mathbb Z \\ n_2\ne\pm n_1}}\frac{\Gamma_{n_1-n_2}}{\langle\vert n_1\vert - \vert n_2\vert\rangle^s} 
	\le\bigg(\sum_{\substack{n_2\in\mathbb Z \\ n_2\ne\pm n_1}}\langle n_1-n_2\rangle^{2\lambda}\Gamma^2_{n_1-n_2}\bigg)^{1/2} \\
	\times\bigg(\sum_{\substack{n_2\in\mathbb Z \\ n_2\ne\pm n_1}}\frac1{\langle\vert n_1\vert - \vert n_2\vert\rangle^{2s}\langle n_1-n_2\rangle^{2\lambda}}\bigg)^{1/2}.
\end{multline*}
There are now two terms to treat. On the one hand, it follows from the definition \eqref{18032024E5} of the $\Gamma_n$ and Minkowski's integral inequality that
\begin{align*}
	\bigg(\sum_{n\in\mathbb Z}\langle n\rangle^{2\lambda}\Gamma^2_n\bigg)^{1/2}
	& = \bigg(\sum_{n\in\mathbb Z}\langle n\rangle^{2\lambda}\Vert\langle\tau\rangle^s\widehat{c_n}(\tau)\Vert^2_{L^{\infty}(\mathbb R)}\bigg)^{1/2} \\
	& \le \bigg(\sum_{n\in\mathbb Z}\langle n\rangle^{2\lambda}\Vert\langle D_t\rangle^sc_n\Vert^2_{L^1(\mathbb R)}\bigg)^{1/2} \\[5pt]
	& = \Vert\langle n\rangle^{2\lambda}\langle D_t\rangle^sc_n\Vert_{l^2_nL^1_t}
	\le \Vert\langle n\rangle^{2\lambda}\langle D_t\rangle^sc_n\Vert_{L^1_tl^2_n}
	= \Vert\varphi\Vert_{W^{s,1}(\mathbb R, H^{\lambda}(\mathbb T))}.
\end{align*}
On the other hand, the second term can be bounded by noticing that $\vert\vert n_1\vert-\vert n_2\vert\vert\le\vert n_1-n_2\vert$ and writing
\begin{align*}
	\sum_{\substack{n_2\in\mathbb Z \\ n_2\ne\pm n_1}}\frac1{\langle\vert n_1\vert - \vert n_2\vert\rangle^{2s}\langle n_1-n_2\rangle^{2\lambda}}
	\le\sum_{n_2\in\mathbb Z}\frac1{\langle\vert n_1\vert - \vert n_2\vert\rangle^{2(s+\lambda)}}
	\lesssim\sum_{l\in\mathbb Z}\frac1{\langle l\rangle^{2(\lambda+s)}}.
\end{align*}
Since $\lambda+s>1/2$, the above sum is finite. We therefore deduce that the term $A_1$ is bounded as follows
$$A_1\lesssim_{\lambda,s}\Vert\varphi\Vert_{W^{s,1}(\mathbb R, H^{\lambda}(\mathbb T))}.$$
A similar estimate holds for the other term $A_2$, with a similar proof. In a nutshell, we get that
\begin{equation}\label{09012025E1}
	\sum_{\substack{n+ n_1 - n_2 = 0,\\ n\ne0,\,n_1\ne-n_2}}\vert\widehat{c_n}(\vert n_2\vert^{\alpha} - \vert n_1\vert^{\alpha})\vert\vert a_{n_1}\vert\vert b_{n_2}\vert 
	\lesssim_{s,\lambda}\Vert\varphi\Vert_{W^{s,1}(\mathbb R, H^{\lambda}(\mathbb T))}\Vert u_0\Vert_{H^{-\sigma}(\mathbb T)}\Vert v_0\Vert_{H^{-\sigma}(\mathbb T)}.
\end{equation}
\textbf{$\triangleright$ Conclusion.} Gathering the estimates \eqref{09012025E2} and \eqref{09012025E1}, we get that
\begin{align*}
	&\ \bigg\vert\int_{\mathbb R}\int_{\mathbb T}(\Pi_0^{\perp}\varphi)(t,x)(e^{it\vert D_x\vert^{\alpha}}u_0)(x)\overline{(e^{it\vert D_x\vert^{\alpha}}v_0)(x)}\,\mathrm dx\bigg\vert \\[5pt]
	\lesssim_{s,\lambda} &\ (\Vert\varphi\Vert_{W^{s,1}(\mathbb R,H^{2\sigma}(\mathbb T))} + \Vert\varphi\Vert_{W^{s,1}(\mathbb R, H^{\lambda}(\mathbb T))})\Vert u_0\Vert_{H^{-\sigma}(\mathbb T)}\Vert v_0\Vert_{H^{-\sigma}(\mathbb T)}.
\end{align*}
Recall that $\lambda\geq0$ stands for any non-negative real number satisfying $\lambda+s>1/2$. Notice that in the case where $\lambda>2\sigma$, which occurs when 
$$\frac12-s\geq2\sigma\ \Longleftrightarrow\ \sigma\le\frac14 - \frac1{4\alpha},$$
we have
$$\Vert\varphi\Vert_{W^{s,1}(\mathbb R, H^{2\sigma}(\mathbb T))}\le\Vert\varphi\Vert_{W^{s,1}(\mathbb R, H^{\lambda}(\mathbb T))}.$$
Moreover, this inequality reverses in the situation where $\lambda\le2\sigma$. The proof of Proposition \ref{prop:ineqtest} is now ended.

\subsection{Optimality} To end this section, let us discuss the optimality of the exponents $s_1$ and $s_2$ appearing in Theorem \ref{thm:se}. Precisely, in the following three situations
$$s_1<\dfrac{2\sigma}{\alpha-1}\quad  \&\quad  s_2\in\mathbb R,\quad   s_1\in\mathbb R\quad  \&\quad  s_2<2\sigma,\quad s_1 = \dfrac{2\sigma}{\alpha-1}\quad  \&\quad  s_2\le\dfrac12 - \dfrac{2\sigma}{\alpha-1},$$
we will construct a function $u_0\in L^2(\mathbb T)$ such that the associated sequence \eqref{22022024E3} does not converge in the space $W^{-s_1,\infty}(\mathbb R, H^{-s_2}(\mathbb T))$. Such functions $u_0\in L^2(\mathbb T)$ will be chosen in order to ensure that
$$\Vert f_N\Vert_{W^{-s_1,\infty}(\mathbb R, H^{-s_2}(\mathbb T))}\geq\Vert\langle D_t\rangle^{-s_1}f_N(0)\Vert_{H^{-s_2}(\mathbb T)}\underset{N\rightarrow+\infty}{\longrightarrow}+\infty,$$
where $f_N$ denotes the elements of the sequence \eqref{22022024E3}. Notice that each function $f_N$ is continuous with respect to the time variable $t\in\mathbb R$. Evaluating at $t=0$ therefore makes sense.

Let us make some preliminary computations, in order to derive a manageable formula for the norm $\Vert\langle D_t\rangle^{-s_1}f_N(0)\Vert_{H^{-s_2}(\mathbb T)}$, given a function $u_0\in L^2(\mathbb T)$. In order to simplify the writing, the Fourier coefficients of the function $u_0\in L^2(\mathbb T)$ will be denoted by
$$a_n = \widehat{u_0}(n),\quad n\in\mathbb Z.$$
Recall from \eqref{22022024E3} that the functions $f_N$ are given by
$$f_N(t,x) = \big\vert\langle D_x\rangle^{\sigma}(e^{it\vert D_x\vert^{\alpha}}\Pi_Nu_0)\big\vert^2 - \Vert\Pi_Nu_0\Vert^2_{H^{\sigma}(\mathbb T)},\quad (t,x)\in\mathbb R\times\mathbb T.$$
By first developing the squares, and then rearranging the terms while denoting $n = n_1-n_2$ and $m = n_2$, we get the following expression for the functions $f_N$
\begin{align*}
	f_N(t,x) 
	& = \sum_{\substack{\vert n_1\vert,\vert n_2\vert\le N \\ n_1\ne n_2}}\langle n_1\rangle^{\sigma}\langle n_2\rangle^{\sigma}a_{n_1}\overline{a_{n_2}}e^{it(\vert n_1\vert^{\alpha}-\vert n_2\vert^{\alpha})}e^{i(n_1-n_2)x} \\
	& = \sum_{0<\vert n\vert\le2N}e^{inx}\bigg(\sum_{\vert m\vert,\vert n+m\vert\le N}\langle n+m\rangle^{\sigma}\langle m\rangle^{\sigma}a_{n+m}\overline{a_m}\,e^{it(\vert n+m\vert^{\alpha}-\vert m\vert^{\alpha})}\bigg).
\end{align*}
We therefore deduce that
\begin{multline}\label{07022025E1}
	\Vert\langle D_t\rangle^{-s_1}f_N(t)\Vert_{H^{-s_2}(\mathbb T)}^2 \\
	= \sum_{0<\vert n\vert\le2N}\frac1{\langle n\rangle^{2s_2}}\bigg\vert\sum_{\vert m\vert,\vert n+m\vert\le N}\frac{\langle n+m\rangle^{\sigma}\langle m\rangle^{\sigma}}{\langle\vert n+m\vert^{\alpha}-\vert m\vert^{\alpha}\rangle^{s_1}}a_{n+m}\overline{a_m}\,e^{it(\vert n+m\vert^{\alpha}-\vert m\vert^{\alpha})}\bigg\vert^2,
\end{multline}
and the following formula at time $t=0$
\begin{multline}\label{24012025E1}
	\Vert\langle D_t\rangle^{-s_1}f_N(0)\Vert_{H^{-s_2}(\mathbb T)}^2 \\
	= \sum_{0<\vert n\vert\le2N}\frac1{\langle n\rangle^{2s_2}}\bigg\vert\sum_{\vert m\vert,\vert n+m\vert\le N}\frac{\langle n+m\rangle^{\sigma}\langle m\rangle^{\sigma}}{\langle\vert n+m\vert^{\alpha}-\vert m\vert^{\alpha}\rangle^{s_1}}a_{n+m}\overline{a_m}\bigg\vert^2.
\end{multline}
\textbf{$\triangleright$ Regularity in time.} Let us now focus on the optimality of the time-regularity, that is, on the optimality of the exponent $s_1$. To that end, let us consider $\varepsilon>0$ and make the following choice for the Fourier coefficients $a_n$
$$a_n = \frac1{\langle n\rangle^{1/2+\varepsilon}},\quad n\in\mathbb Z.$$
By keeping only the term $n=1$ in the first sum in \eqref{24012025E1}, and restricting the second one to positive integers $m\geq1$, we deduce that
$$\Vert\langle D_t\rangle^{-s_1}f_N(0)\Vert_{H^{-s_2}(\mathbb T)}^2\gtrsim_{s_2}\bigg\vert\sum_{1\le m\le N-1}\frac{\langle 1+m\rangle^{\sigma}\langle m\rangle^{\sigma}}{\langle\vert 1+m\vert^{\alpha}-\vert m\vert^{\alpha}\rangle^{s_1}}\frac1{\langle 1+m\rangle^{1/2+\varepsilon}}\frac1{\langle m\rangle^{1/2+\varepsilon}}\bigg\vert^2.$$
We have to control the different terms appearing in the above sum. First of all, notice that the following estimates hold
$$\langle1+m\rangle\lesssim\langle m\rangle\le\langle1+m\rangle.$$
Moreover, by using again the formula \eqref{24012025E2}, we get that
$$0\le\vert 1+m\vert^{\alpha}-\vert m\vert^{\alpha}\lesssim_{\alpha}(1+m)^{\alpha-1},$$
which leads to the following estimate
$$\langle\vert 1+m\vert^{\alpha}-\vert m\vert^{\alpha}\rangle\lesssim_{\alpha}\langle1+m\rangle^{\alpha-1}.$$
We therefore deduce that
$$\Vert\langle D_t\rangle^{-s_1}f_N(0)\Vert_{H^{-s_2}(\mathbb T)}^2\gtrsim_{\alpha,s_1,s_2}\bigg\vert\sum_{1\le m\le N-1}\frac1{\langle1+m\rangle^{s_1(\alpha-1)+1+2\varepsilon-2\sigma}}\bigg\vert^2.$$
This proves that under the condition
$$s_1(\alpha-1)+1+2\varepsilon-2\sigma\le1\quad\Longleftrightarrow\quad s_1\le\frac{2\sigma}{\alpha-1} - \frac{2\varepsilon}{\alpha-1},$$
we have
$$\lim_{N\rightarrow+\infty}\Vert\langle D_t\rangle^{-s_1}f_N(0)\Vert_{H^{-s_2}(\mathbb T)}^2 = +\infty,\quad\text{whatever $s_2\in\mathbb R$},$$
and the sequence \eqref{22022024E3} therefore does not converge in the space $W^{-s_1,\infty}(\mathbb R, H^{-s_2}(\mathbb T))$. \\[2pt]
\noindent\textbf{$\triangleright$ Regularity in space.} Let us now focus on the optimality of the exponent $s_2$. Assuming first that $s_2<2\sigma$, we make the following choice for the Fourier coefficients $a_n$
$$a_n = a_{-n} = \left\{\begin{array}{cl}
	\dfrac1{n^{(2\sigma - s_2)/2}} & \text{when $n = 2^k$ for some $k\in\mathbb N$,} \\[10pt]
	0 & \text{otherwise}.
\end{array}\right.$$
Notice that the sequence $(a_n)_n$ indeed belongs to $\ell^2(\mathbb Z)$, due to the fact that since $2\sigma-s_2>0$, we have
$$\sum_{n = 0}^{+\infty}\vert a_n\vert^2 = \sum_{k = 0}^{+\infty}\vert a_{2^k}\vert^2 = \sum_{k = 0}^{+\infty}\bigg(\frac1{2^{2\sigma-s_2}}\bigg)^k<+\infty.$$
Let us come back to the expression \eqref{24012025E1}, and keep only the even terms in the first sum in order to get
\begin{multline*}
	\Vert\langle D_t\rangle^{-s_1}f_N(0)\Vert_{H^{-s_2}(\mathbb T)}^2 \\
	\geq \sum_{0<\vert n\vert\le N}\frac1{\langle 2n\rangle^{2s_2}}\bigg\vert\sum_{\vert m\vert,\vert 2n+m\vert\le N}\frac{\langle 2n+m\rangle^{\sigma}\langle m\rangle^{\sigma}}{\langle\vert 2n+m\vert^{\alpha}-\vert m\vert^{\alpha}\rangle^{s_1}}a_{2n+m}a_m\bigg\vert^2.
\end{multline*}
The strategy is then to keep only the term $m = -n$ in the second sum
$$\Vert\langle D_t\rangle^{-s_1}f_N(0)\Vert_{H^{-s_2}(\mathbb T)}^2
\geq \sum_{0<\vert n\vert\le N}\frac{\langle n\rangle^{4\sigma}}{\langle 2n\rangle^{2s_2}}\vert a_{n}a_{-n}\vert^2
\gtrsim_{s_2} \sum_{n = 1}^N\vert n\vert^{4\sigma - 2s_2}\vert a_{n}a_{-n}\vert^2.$$
The sequence $(a_n)_n$ has been constructed in order to satisfy that
$$\sum_{n = 1}^{2^N}\vert n\vert^{4\sigma - 2s_2}\vert a_{n}a_{-n}\vert^2 = \sum_{n = 1}^{2^N}\vert n\vert^{4\sigma - 2s_2}\vert a_{n}\vert^4
 = \sum_{k=0}^N\frac{\vert 2^k\vert^{4\sigma - 2s_2}}{\vert 2^k\vert^{4\sigma - 2s_2}} = 1+N\underset{N\rightarrow+\infty}{\longrightarrow}+\infty.$$
As a consequence, the sequence \eqref{22022024E3} cannot converge in the space $W^{-s_1,\infty}(\mathbb R, H^{-s_2}(\mathbb T))$.

Finally, let us assume that $s_1 = \frac{2\sigma}{\alpha-1}$ and $2(s_1+s_2) = 1$. Let us also choose the Fourier coefficients $a_n$ being equal to $0$ when $n\le0$, $a_1 = 1$ and otherwise
$$a_n = \frac1{n^{1/2}(\log n)^{3/4}},\quad n\geq 2.$$
By using again the expression \eqref{24012025E1} for $2N$, while restricting the first sum to the indices $-N\le n\le -1$ and restricting the second one to $-2n\le m\le 2N$, we get that
\begin{align*}
	\Vert\langle D_t\rangle^{-s_1}f_{2N}(0)\Vert_{H^{-s_2}(\mathbb T)}^2
	& \geq \sum_{n=-N}^{-1}\frac1{\langle n\rangle^{2s_2}}\bigg\vert\sum_{m = -2n}^{2N}\frac{\langle n+m\rangle^{\sigma}\langle m\rangle^{\sigma}}{\langle\vert n+m\vert^{\alpha}-\vert m\vert^{\alpha}\rangle^{s_1}}a_{n+m}a_m\bigg\vert^2 \\
	& = \sum_{n=1}^{N}\frac1{\langle n\rangle^{2s_2}}\bigg\vert\sum_{m = 2n}^{2N}\frac{\langle m-n\rangle^{\sigma}\langle m\rangle^{\sigma}}{\langle\vert m-n\vert^{\alpha}-\vert m\vert^{\alpha}\rangle^{s_1}}a_{m-n}a_m\bigg\vert^2.
\end{align*}
We now have to control the terms in the above sum. On the one hand, since $0\le m-n\le m$ and that the sequence $(a_n)_{n\geq1}$ is decreasing, we get that
$$\langle m-n\rangle^{\sigma}\langle m\rangle^{\sigma}\geq\langle m-n\rangle^{2\sigma}\quad\text{and}\quad a_{m-n}a_m\geq a_m^2.$$
On the other hand, since $m\geq n$, it follows from the formula \eqref{24012025E2} that
$$\vert(m-n)^{\alpha} - m^{\alpha}\vert\lesssim_{\alpha}(m-n + m)^{\alpha-1}\vert m-n-m\vert = (2m-n)^{\alpha-1}n.$$
By using the fact that $n\geq1$, we therefore deduce that
\begin{align*}
	\langle\vert m-n\vert^{\alpha} - \vert m\vert^{\alpha}\rangle^2 & \lesssim_{\alpha} 1+(2m-n)^{2(\alpha-1)}n^2 \\[5pt]
	& \lesssim_{\alpha} (1+(2m-n)^{2(\alpha-1)})n^2 \\[5pt]
	& \lesssim_{\alpha}\langle2m-n\rangle^{2(\alpha-1)}\langle n\rangle^2.
\end{align*}
The restriction $m\geq2n$ allows to get the following bound $2m-n\le3(m-n)$, and as a consequence,
$$\langle\vert m-n\vert^{\alpha} - \vert m\vert^{\alpha}\rangle^{s_1}\lesssim_{\alpha,s_1}\langle m-n\rangle^{s_1(\alpha-1)}\langle n\rangle^{s_1}.$$
Gathering all these estimates, and recalling that  $s_1 = \frac{2\sigma}{\alpha-1}$ and $2(s_1+s_2) = 1$, we obtain
\begin{align*}
	\Vert\langle D_t\rangle^{-s_1}f_{2N}(0)\Vert_{H^{-s_2}(\mathbb T)}^2
	& \gtrsim_{\alpha,s_1}\sum_{n=1}^N\frac1{\langle n\rangle^{2(s_1+s_2)}}\bigg\vert\sum_{m = 2n}^{2N}\frac{\langle m-n\rangle^{2\sigma}}{\langle m-n\rangle^{s_1(\alpha-1)}}a_m^2\bigg\vert^2 \\[5pt]
	& \gtrsim \sum_{n=1}^N\frac1n\bigg\vert\sum_{m = 2n}^{2N}a_m^2\bigg\vert^2.
\end{align*}
The above second sum will be controlled by using a comparison with an integral, the latter being computed with the change of variable $u = \log x$,
$$\sum_{m = 2n}^{2N}a_m^2 = \sum_{m = 2n}^{2N}\frac1{m(\log m)^{3/2}}\geq\int_{2n}^{2N}\frac{\mathrm dx}{x(\log x)^{3/2}} = \frac12\bigg(\frac1{\sqrt{\log(2n)}} - \frac1{\sqrt{\log(2N)}}\bigg).$$
As a consequence,
$$\Vert\langle D_t\rangle^{-s_1}f_{2N}(0)\Vert_{H^{-s_2}(\mathbb T)}^2\gtrsim_{\alpha,s_1}\sum_{n=1}^{N}\frac1n\bigg\vert\frac1{\sqrt{\log(2n)}} - \frac1{\sqrt{\log(2N)}}\bigg\vert^2.$$
We now have three numerical series to study, coming from an expansion of the above square. First of all, we classically have
$$\sum_{n=1}^{N}\frac1n\frac1{\log(2n)}\underset{N\rightarrow+\infty}{\longrightarrow}+\infty\quad\text{and}\quad\frac1{\log(2N)}\sum_{n=1}^{N}\frac1n\underset{N\rightarrow+\infty}{\longrightarrow}1.$$
The last one is also compared to an integral as follows
$$\frac1{\sqrt{\log(2N)}}\sum_{n=2}^{N}\frac1n\frac1{\sqrt{\log(2n)}}\le\frac1{\sqrt{\log(2N)}}\int_1^N\frac{\mathrm dx}{x\sqrt{\log(2x)}} \lesssim \frac{\sqrt{\log(2N)}}{\sqrt{\log(2N)}} = 1.$$
We therefore deduce that
$$\lim_{N\rightarrow+\infty}\Vert\langle D_t\rangle^{-s_1}f_{2N}(0)\Vert_{H^{-s_2}(\mathbb T)}^2 = +\infty,$$
and the sequence \eqref{22022024E3} cannot converge in the space $W^{-s_1,\infty}(\mathbb R, H^{-s_2}(\mathbb T))$. Notice that as a consequence, the very same holds when $2(s_1+s_2) < 1$ since we have $H^{-s_2}(\mathbb T)\hookrightarrow H^{-1/2+s_1}(\mathbb T)$ in this situation.

The discussion of the optimality of the exponents $s_1$ and $s_2$ in Theorem \ref{thm:se} is now ended. Let us finally mention that by using the formula \eqref{07022025E1}, one could actually give a direct proof of Theorem \ref{thm:se}. However, it would not allow to derive the bilinear estimates stated in Proposition \ref{prop:ineqtest}, which is a stronger result being moreover key in the proof of the observability result stated in Theorem \ref{thm:fracobs}.

\section{Observability estimates}

The aim of this section is to prove Theorem \ref{thm:fracobs}. 
Let $\alpha>1$ be a positive real number and $b\in L^1(\mathbb T)\setminus\{0\}$ be a non-negative function.
The strategy consists in four steps:
\begin{enumerate}[leftmargin=* ,parsep=2pt,itemsep=0pt,topsep=2pt]
	\item[1.] Proof of the following $L^{\infty}_xL^2_t$ Strichartz estimate
	$$\big\Vert e^{it\vert D_x\vert^{\alpha}}u_0\big\Vert_{L^{\infty}(\mathbb T,L^2(0,T))}\le C_T\Vert u_0\Vert_{L^2(\mathbb T)}.$$
	\item[2.] Proof of the following weak observability estimate
	$$\Vert u_0\Vert^2_{L^2(\mathbb T)}\le C_{b,T}\bigg(\int_0^T\int_{\mathbb T}b(x)\big\vert(e^{it\vert D_x\vert^{\alpha}}u_0)(x)\big\vert^2\,\mathrm dx\,\mathrm dt + \Vert u_0\Vert^2_{H^{-\alpha}(\mathbb T)}\bigg).$$
	\item[3.] Proof of the following unique continuation result
	$$\text{\big($\sqrt be^{it\vert D_x\vert^{\alpha}}u_0 = 0$ on $(0,T)\times\mathbb T$}\big)\quad\Longrightarrow\quad\text{$u_0\equiv0$ on $\mathbb T$}.$$
	\item[4.] Removing the $H^{-\alpha}$-norm in the above weak observability estimate.
\end{enumerate}
Steps 2,3 and 4 are now very classical and had been used to study the observability properties of Schr\"odinger-type equations on tori, see \cite{Lau} for a review.
Steps 3 and 4 are known as the Bardos--Lebeau--Rauch uniqueness-compactness argument. The $L^{\infty}_xL^2_t$ Strichartz estimate presented in Step 1, crucial for dealing with the case $b\in L^2(\mathbb T)$, has been first introduced in the paper \cite{BBZ} in the non-fractional case $\alpha = 2$.

\subsection{Step 1: Strichartz type estimate} As a preliminary, let us first prove a $L^{\infty}_xL^2_t$ estimate for the solutions of the equation \eqref{eq:schro}, which is an adaptation of \cite[Proposition 2.1]{BBZ} to the fractional case.  A key feature of this estimate  is that it does not encounter derivative losses as does e.g. the $L^4_{t,x}$ Strichartz estimates for 
$e^{it\vert D_x\vert^{\alpha}}$ obtained in \cite{ST}. 
\begin{lem}\label{11032024L1} For every positive time $T>0$, there exists a positive constant $C_T>0$ such that for every initial datum $u_0\in L^2(\mathbb T)$,
$$\big\Vert e^{it\vert D_x\vert^{\alpha}}u_0\big\Vert_{L^{\infty}(\mathbb T,L^2(0,T))}\le C_T\Vert u_0\Vert_{L^2(\mathbb T)}.$$
\end{lem}

\begin{proof} Setting $c_n = \widehat{u_0}(n)$, we begin by developing the $L^{\infty}_xL^2_t$ norm we aim at estimating
\begin{align*}
	\big\Vert e^{it\vert D_x\vert^{\alpha}}u_0\big\Vert^2_{L^{\infty}(\mathbb T,L^2(0,T))}
	& = \sup_{x\in\mathbb T}\int_0^T\bigg\vert\sum_{n\in\mathbb Z}c_ne^{it\vert n\vert^{\alpha}}e^{inx}\bigg\vert^2\mathrm dt \\
	& = \sup_{x\in\mathbb T}\sum_{n_1,n_2\in\mathbb Z}\bigg(\int_0^Te^{it(\vert n_1\vert^{\alpha}-\vert n_2\vert^{\alpha})}\,\mathrm dt\bigg)e^{i(n_1-n_2)x}c_{n_1}\overline{c_{n_2}}.
\end{align*}
There are then two cases to consider for the integers $n_1$ and $n_2$ in the above sum. \\[2pt]
\textbf{$\triangleright$ Resonant case.} In the situation where $n_1 = \pm n_2$, we directly get that
$$\sup_{x\in\mathbb T}\sum_{\substack{n_1,n_2\in\mathbb Z \\ n_1 = \pm n_2}}\bigg(\int_0^Te^{it(\vert n_1\vert^{\alpha}-\vert n_2\vert^{\alpha})}\,\mathrm dt\bigg)e^{i(n_1-n_2)x}c_{n_1}\overline{c_{n_2}}
\lesssim_T\sum_{n_1\in\mathbb Z}\vert c_{n_1}\vert^2 = \Vert u_0\Vert^2_{L^2(\mathbb T)}.$$
\textbf{$\triangleright$ Non-resonant case.} In the case where $n_1\ne\pm n_2$, we begin by writing that
$$\sup_{x\in\mathbb T}\sum_{\substack{n_1,n_2\in\mathbb Z \\ n_1\ne\pm n_2}}\bigg(\int_0^Te^{it(\vert n_1\vert^{\alpha}-\vert n_2\vert^{\alpha})}\,\mathrm dt\bigg)e^{i(n_1-n_2)x}c_{n_1}\overline{c_{n_2}}
\le\sum_{\substack{n_1,n_2\in\mathbb Z \\ n_1\ne\pm n_2}}\frac{\vert c_{n_1}\vert\vert c_{n_2}\vert}{\vert\vert n_1\vert^{\alpha}-\vert n_2\vert^{\alpha}\vert}.$$
By using anew Lemma \ref{13032024L1}, we obtain the following bound for the above denominator
$$\vert\vert n_1\vert^{\alpha} - \vert n_2\vert^{\alpha}\vert\gtrsim(\vert n_1\vert + \vert n_2\vert)^{\alpha-1}\vert\vert n_1\vert - \vert n_2\vert\vert\geq\vert\vert n_1\vert - \vert n_2\vert\vert^{\alpha}.$$
Cauchy-Schwarz' inequality then implies that
\begin{align*}
	\sum_{\substack{n_1,n_2\in\mathbb Z \\ n_1\ne\pm n_2}}\frac{\vert c_{n_1}\vert\vert c_{n_2}\vert}{\vert\vert n_1\vert^{\alpha}-\vert n_2\vert^{\alpha}\vert}
	& \lesssim\sum_{\substack{n_1,n_2\in\mathbb Z \\ n_1\ne\pm n_2}}\frac{\vert c_{n_1}\vert\vert c_{n_2}\vert}{\vert\vert n_1\vert - \vert n_2\vert\vert^{\alpha}} \\
	& = \sum_{\substack{n_1,n_2\geq0 \\ n_1\ne n_2}}\frac{\vert c_{n_1}\vert\vert c_{n_2}\vert+\vert c_{-n_1}\vert\vert c_{n_2}\vert+\vert c_{n_1}\vert\vert c_{-n_2}\vert+\vert c_{-n_1}\vert\vert c_{-n_2}\vert}{\vert n_1 - n_2\vert^{\alpha}} \\
	& \lesssim\bigg(\sum_{l\in\mathbb N\setminus\{0\}}\frac1{l^{\alpha}}\bigg)\Vert u_0\Vert^2_{L^2(\mathbb T)}.
\end{align*}
Notice that the above sum is finite since $\alpha>1$. This ends the proof of Lemma \ref{11032024L1}.
\end{proof}

\subsection{Step 2: estimate with error term} 
We now derive the following weak observability estimates from the smoothing estimates stated in Proposition \ref{prop:ineqtest}.

\begin{lem}\label{lem:weakobs} For every positive time $T>0$, there exists a positive constant $C_{b,T}>0$ such that for every $u_0\in L^2(\mathbb T)$,
$$\Vert u_0\Vert^2_{L^2(\mathbb T)}\le C_{b,T}\bigg(\int_0^T\int_{\mathbb T}b(x)\big\vert(e^{it\vert D_x\vert^{\alpha}}u_0)(x)\big\vert^2\,\mathrm dx\,\mathrm dt + \Vert u_0\Vert^2_{H^{-\alpha}(\mathbb T)}\bigg).$$
\end{lem}

\begin{proof} Arguing by contradiction, we consider a sequence $(u_{n,0})_n$ in $L^2(\mathbb T)$ such that $\Vert u_{n,0}\Vert_{L^2(\mathbb T)} = 1$ and satisfying
\begin{equation}\label{18022024E1}
	\int_0^T\int_{\mathbb T}b(x)\big\vert(e^{it\vert D_x\vert^{\alpha}}u_{n,0})(x)\big\vert^2\,\mathrm dx\,\mathrm dt + \Vert u_{n,0}\Vert^2_{H^{-\alpha}(\mathbb T)}\le\frac1n.
\end{equation}
Setting $u_n(t) = e^{it\vert D_x\vert^{\alpha}}u_{n,0}$, 
we aim at proving that the sequence $(u_n)_n$ 
satisfies that for all $\psi\in C^{\infty}_0(0,T)$
\begin{equation}\label{18022024E2}
	\int_0^T\psi(t)\Vert u_n(t)\Vert^2_{L^2(\mathbb T)}\,\mathrm dt\underset{n\rightarrow+\infty}{\longrightarrow}0.
\end{equation}
This would end the proof of Lemma \ref{lem:weakobs}. Indeed, the conservation of mass of the equation \eqref{eq:schro} would then imply that
$$\int_0^T\psi(t)\Vert u_n(t)\Vert^2_{L^2(\mathbb T)}\,\mathrm dt = \int_0^T\psi(t)\Vert u_{n,0}\Vert^2_{L^2(\mathbb T)}\,\mathrm dt = \int_0^T\psi(t)\,\mathrm dt\underset{n\rightarrow+\infty}{\longrightarrow}0.$$
This convergence of course cannot hold for all functions $\psi\in C^{\infty}_0(0,T)$. \\[2pt]
\textbf{$\triangleright$ Step 2.1.} First of all, let us prove that for every $x_0\in\mathbb T$,
\begin{equation}\label{11032024E3}
	\int_0^T\int_{\mathbb T}\psi(t)(\tau_{x_0}b)(x)\vert u_n(t,x)\vert^2\,\mathrm dx\,\mathrm dt\underset{n\rightarrow+\infty}{\longrightarrow}0,
\end{equation}
where the translation operator $\tau_{x_0}$ is defined for every function in $f\in L^1(\mathbb T)$ by
$$(\tau_{x_0}f)(x) = f(x-x_0),\quad x\in\mathbb T.$$
To that end, we will widely use the fact that for every smooth function $\varphi\in C^{\infty}(\mathbb T)$, we get from the smoothing estimates stated in Proposition \ref{prop:ineqtest} applied with $\sigma=\alpha$, $u_0 = v_0 = u_{n,0}$ and the tensorized test function $\psi\otimes\varphi$ that
\begin{equation}\label{11032024E2}
	\bigg\vert\int_0^T\int_{\mathbb T}\psi(t)(\Pi_0^{\perp}\varphi)(x)\vert u_n(t,x)\vert^2\,\mathrm dx\,\mathrm dt\bigg\vert
	\lesssim_{\psi,\varphi}\Vert u_{n,0}\Vert^2_{H^{-\alpha}(\mathbb T)}\underset{n\rightarrow+\infty}{\longrightarrow}0\quad\text{by \eqref{18022024E1}}.
\end{equation}
Let us consider a sequence $(b_j)_j$ in $C^{\infty}(\mathbb T)$ that converges to the function $b$ in $L^1(\mathbb T)$, and write
\begin{multline*}
	\bigg\vert\int_0^T\int_{\mathbb T}\psi(t)(b-\tau_{x_0}b)(x)\vert u_n(t,x)\vert^2\,\mathrm dx\,\mathrm dt\bigg\vert 
	\le\bigg\vert\int_0^T\int_{\mathbb T}\psi(t)(b_j-\tau_{x_0}b_j)(x)\vert u_n(t,x)\vert^2\,\mathrm dx\,\mathrm dt\bigg\vert \\
	+ \bigg\vert\int_0^T\int_{\mathbb T}\psi(t)(b-b_j-\tau_{x_0}(b-b_j))(x)\vert u_n(t,x)\vert^2\,\mathrm dx\,\mathrm dt\bigg\vert.
\end{multline*}
There are now two terms to treat. On the one hand, we deduce from the convergence \eqref{11032024E2} with the functions $\varphi = b_j - \tau_{x_0}b_j$, which satisfy $\Pi^{\perp}_0\varphi = \varphi$, that the following convergence holds for every integer $j\geq0$
$$\int_0^T\int_{\mathbb T}\psi(t)(b_j-\tau_{x_0}b_j)(x)\vert u_n(t,x)\vert^2\,\mathrm dx\,\mathrm dt\underset{n\rightarrow+\infty}{\longrightarrow}0.$$
On the other hand, we deduce from the $L^{\infty}_xL^2_t$ Strichartz estimate stated in Lemma \ref{11032024L1} that
$$\bigg\vert\int_0^T\int_{\mathbb T}\psi(t)(b-b_j-\tau_{x_0}(b-b_j))(x)\vert u_n(t,x)\vert^2\,\mathrm dx\,\mathrm dt\bigg\vert\lesssim_{\psi}\Vert b - b_j\Vert_{L^1(\mathbb T)},$$
since the functions $u_{n,0}$ are normalized in $L^2(\mathbb T)$. We therefore deduce that
$$\limsup_{n\rightarrow+\infty}\bigg\vert\int_0^T\int_{\mathbb T}\psi(t)(b-\tau_{x_0}b)(x)\vert u_n(t,x)\vert^2\,\mathrm dx\,\mathrm dt\bigg\vert
\lesssim\Vert b - b_j\Vert_{L^1(\mathbb T)}\underset{j\rightarrow+\infty}{\longrightarrow}0.$$
Recalling from \eqref{18022024E1} that we also have
$$\int_0^T\int_{\mathbb T}\psi(t)b(x)\vert u_n(t,x)\vert^2\,\mathrm dx\,\mathrm dt\underset{n\rightarrow+\infty}{\longrightarrow}0,$$
this proves that the convergence \eqref{11032024E3} actually holds. Notice that in the very particular case where $b=\mathbbm1_{\omega}$ and $\omega\subset\mathbb T$ is an open set, we directly conclude to the convergence \eqref{18022024E2} by covering the circle $\mathbb T$ with a finite number of open sets of the form $\tau_{x_0}\omega$ with $x_0\in\mathbb T$. The $L^{\infty}_xL^2_t$ Strichartz estimate is in fact not necessary in this case, since it suffices to apply the convergence \eqref{11032024E2} with the functions $\varphi = f - \tau_{x_0}f$, where $f\in C^{\infty}(\mathbb T)$ is a smooth function such that $\supp f\subset\omega$. \\[2pt]
\textbf{$\triangleright$ Step 2.2.} In order to derive the convergence \eqref{18022024E2} from \eqref{11032024E3} when $b\in L^1(\mathbb T)$, we need to establish the following approximation result.

\begin{lem}\label{11032024L2} For every function $f\in L^1(\mathbb T)$, we have
$$1\in\overline{\vect\{\tau_{x_0}f : x_0\in\mathbb T\}}\quad\Longleftrightarrow\quad \widehat f(0)\ne0,$$
the closure being taken with respect to the $L^1(\mathbb T)$ topology.
\end{lem}

\begin{proof} The necessary part of this statement follows from the fact that 
$$\forall x_0\in\mathbb T,\quad\widehat{\tau_{x_0}f}(0) = \int_{\mathbb T}(\tau_{x_0}f)(x)\,\mathrm dx = \int_{\mathbb T}f(x)\,\mathrm dx = \widehat f(0).$$
Concerning the sufficient part, we can assume without lost of generality that $\widehat f(0) = 1$. Let us consider $\alpha\in\mathbb R\setminus2\pi\mathbb Q$. In order to simplify the writing, we will use the following notation for every positive integer $n\geq1$ and $g\in L^1(\mathbb T)$
$$S_n(g) = \frac1n\sum_{k=0}^{n-1}\tau_{k\alpha}g\in\vect\{\tau_{x_0}g : x_0\in\mathbb T\}.$$
We consider a sequence $(f_j)_j$ in $L^2(\mathbb T)$ that converges to the function $f$ in $L^1(\mathbb T)$, and we write
\begin{align*}
	\Vert1-S_n(f)\Vert_{L^1(\mathbb T)}
	&  \le \Vert1-S_n(f_j)\Vert_{L^1(\mathbb T)} + \Vert S_n(f_j)-S_n(f)\Vert_{L^1(\mathbb T)} \\[5pt]
	& \le \Vert1-S_n(f_j)\Vert_{L^2(\mathbb T)} + \frac1n\sum_{k=0}^{n-1}\Vert\tau_{k\alpha}f_j - \tau_{k\alpha}f\Vert_{L^1(\mathbb T)} \\[5pt]
	& \le \vert 1-\widehat{f_j}(0)\vert +  \Vert\widehat{f_j}(0)-S_n(f_j)\Vert_{L^2(\mathbb T)} + \Vert f_j - f\Vert_{L^1(\mathbb T)}.
\end{align*}
Moreover, a classical consequence of von Neumann's ergodic theorem \cite[Theorem 5.1]{SS} implies that for every integer $j\geq0$,
$$\Vert\widehat{f_j}(0)-S_n(f_j)\Vert_{L^2(\mathbb T)}\underset{n\rightarrow+\infty}{\longrightarrow}0.$$
This implies that
$$\limsup_{n\rightarrow+\infty}\Vert1-S_n(f)\Vert_{L^1(\mathbb T)}
\le \vert 1-\widehat{f_j}(0)\vert + \Vert f_j - f\Vert_{L^1(\mathbb T)}
\le2\Vert f_j - f\Vert_{L^1(\mathbb T)}\underset{j\rightarrow+\infty}{\longrightarrow}0.$$
The proof of Lemma \ref{11032024L2} is therefore ended. It should be noted, however, that in the particular case where $f\in L^2(\mathbb T)$, the above ergodic argument can be simplified to a projection argument. Indeed, let us introduce the following notation
$$F = \overline{\vect\{\tau_{x_0}f : x_0\in\mathbb T\}}\subset L^2(\mathbb T).$$
Let $\varphi\in F$ be the orthogonal projection of the constant function $1$ onto the closed vector space $F$. Since $1-\varphi\in F^{\perp}$, we get in particular that
$$\forall x_0\in\mathbb T,\quad 0 = \langle1-\varphi,\tau_{x_0}f\rangle_{L^2(\mathbb T)} = \sum_{n\in\mathbb Z}\widehat{1-\varphi}(n)\widehat f(n)e^{inx_0}.$$
This formula shows that the function $x_0\mapsto\langle1-\varphi,\tau_{x_0}f\rangle_{L^2(\mathbb T)}$ belongs to the space $C(\mathbb T)$. We therefore deduce that
$$\forall n\in\mathbb Z,\quad\widehat{1-\varphi}(n)\widehat f(n) = 0.$$
Since $\widehat f(0)\ne0$, we get that $\widehat\varphi(0) =1$. Moreover, $\Vert\varphi\Vert_{L^2(\mathbb T)}\le 1$, and we conclude that $\varphi = 1$, implying that $1\in F$ as expected. We were not able then to pass directly to the case where $f\in L^1(\mathbb T)$ with an approximation argument. This is the reason why we made use of von Neumann's ergodic theorem.
\end{proof}

We can now prove \eqref{18022024E2}. Since the assumptions on the function $b$ imply that $\widehat b(0)\ne0$, we deduce from Lemma \ref{11032024L2} that there exists a sequence $(b_j)_j$ with
$$b_j\in\vect\{\tau_{x_0}b : x_0\in\mathbb T\},$$
satisfying that
$$b_j\underset{j\rightarrow+\infty}{\longrightarrow}1\quad\text{in $L^1(\mathbb T)$}.$$
Let us now write
\begin{multline*}
	\bigg\vert\int_0^T\int_{\mathbb T}\psi(t)\vert u_n(t,x)\vert^2\,\mathrm dx\,\mathrm dt\bigg\vert 
	\le\bigg\vert\int_0^T\int_{\mathbb T}\psi(t)(1-b_j)(x)\vert u_n(t,x)\vert^2\,\mathrm dx\,\mathrm dt\bigg\vert \\
	+ \bigg\vert\int_0^T\int_{\mathbb T}\psi(t)b_j(x)\vert u_n(t,x)\vert^2\,\mathrm dx\,\mathrm dt\bigg\vert.
\end{multline*}
The first term can be treated as follows by using anew Lemma \ref{11032024L1}
$$\bigg\vert\int_0^T\int_{\mathbb T}\psi(t)(1-b_j)(x)\vert u_n(t,x)\vert^2\,\mathrm dx\,\mathrm dt\bigg\vert\lesssim_{\psi}\Vert 1-b_j\Vert_{L^1(\mathbb T)}.$$
Moreover, as a consequence of \eqref{11032024E3}, we get that for every $j\geq0$,
$$\int_0^T\int_{\mathbb T}\psi(t)b_j(x)\vert u_n(t,x)\vert^2\,\mathrm dx\,\mathrm dt\underset{n\rightarrow+\infty}{\longrightarrow}0.$$
We therefore deduce that
$$\limsup_{n\rightarrow+\infty}\bigg\vert\int_0^T\int_{\mathbb T}\psi(t)\vert u_n(t,x)\vert^2\,\mathrm dx\,\mathrm dt\bigg\vert
\lesssim\Vert 1-b_j\Vert_{L^1(\mathbb T)}\underset{j\rightarrow+\infty}{\longrightarrow}0.$$
This ends the proof of Lemma \ref{lem:weakobs}.
\end{proof}

\subsection{Step 3: unique continuation}\label{uc} Let us now consider the following vector space
$$\mathcal N_{b,T} = \big\{u_0\in L^2(\mathbb T) : \text{$\sqrt be^{it\vert D_x\vert^{\alpha}}u_0 = 0$ on $(0,T)\times\mathbb T$}\big\}.$$
In this subsection, we prove the following unique continuation result.

\begin{lem}\label{18022024E4} For every positive time $T>0$, we have
$$\mathcal N_{b,T} = \{0\}.$$
\end{lem}

\begin{proof} Let us begin by checking that the vector space $\mathcal N_{b,T}$ is invariant under the action of the operator $\vert D_x\vert^{\alpha}$, with moreover
\begin{equation}\label{18022024E5}
	\exists c>0, \forall u_0\in\mathcal N_{b,T},\quad \Vert\vert D_x\vert^{\alpha}u_0\Vert_{L^2(\mathbb T)}\le c\Vert u_0\Vert_{L^2(\mathbb T)}.
\end{equation}
Let $u_0\in\mathcal N_{b,T}$ and $\varepsilon\in(0,T)$. We also consider the function $u_{\varepsilon} = \varepsilon^{-1}(e^{i\varepsilon\vert D_x\vert^{\alpha}}u_0-u_0)$. Since $u_{\varepsilon}\in\mathcal N_{b,T/2}$ when $0<\varepsilon<T/2$, we deduce from Lemma \ref{lem:weakobs} applied at time $T/2$ that
\begin{equation}\label{22022024E2}
	\Vert u_{\varepsilon}\Vert_{L^2(\mathbb T)}\le C_{b,T/2}\Vert u_{\varepsilon}\Vert_{H^{-\alpha}(\mathbb T)}.
\end{equation}
Moreover, Plancherel's theorem and the dominated convergence theorem imply that 
$$\Vert u_{\varepsilon}\Vert_{H^{-\alpha}(\mathbb T)}\rightarrow\Vert\vert D_x\vert^{\alpha}\langle D_x\rangle^{-\alpha}u_0\Vert_{L^2(\mathbb T)}\le\Vert u_0\Vert_{L^2(\mathbb T)}\quad \text{as $\varepsilon\rightarrow0^+$}.$$ 
Indeed, on the one hand, the following convergence holds for every $n\in\mathbb Z$,
$$\langle n\rangle^{-\alpha}\vert \widehat{u_{\varepsilon}}(n)\vert 
= \bigg\vert\frac{e^{i\varepsilon\vert n\vert^{\alpha}}\widehat{u_0}(n)-\widehat{u_0}(n)}{\varepsilon\langle n\rangle^{\alpha}}\bigg\vert
= \frac{\vert n\vert^{\alpha}}{\langle n\rangle^{\alpha}}\bigg\vert\frac{e^{i\varepsilon\vert n\vert^{\alpha}}-1}{\varepsilon\vert n\vert^{\alpha}}\bigg\vert\vert\widehat{u_0}(n)\vert
\underset{\varepsilon\rightarrow0^+}{\longrightarrow}\frac{\vert n\vert^{\alpha}}{\langle n\rangle^{\alpha}}\vert\widehat{u_0}(n)\vert.$$
On the other hand, we have the following domination from the mean value theorem for every $\varepsilon>0$ and $n\in\mathbb Z$,
$$\langle n\rangle^{-\alpha}\vert \widehat{u_{\varepsilon}}(n)\vert\le\vert\widehat{u_0}(n)\vert.$$
In particular, the sequence $(u_{\varepsilon})_{\varepsilon}$ is bounded in $H^{-\alpha}(\mathbb T)$. This fact combined with the estimate \eqref{22022024E2} implies that the sequence $(u_{\varepsilon})_{\varepsilon}$ is also bounded in $L^2(\mathbb T)$ and therefore weakly converges in this space up to extracting. Since the sequence $(u_{\varepsilon})_{\varepsilon}$ also converges to $i\vert D_x\vert^{\alpha}u_0$ in $\mathscr S'(\mathbb T)$, the aforementionned weak limit is nothing but the function $i\vert D_x\vert^{\alpha}u_0$, which therefore belongs to $L^2(\mathbb T)$. Plancherel's theorem and the dominated convergence theorem then implies that 
$$\Vert u_{\varepsilon}\Vert_{L^2(\mathbb T)}\rightarrow\Vert\vert D_x\vert^{\alpha}u_0\Vert_{L^2(\mathbb T)}\quad\text{as $\varepsilon\rightarrow0^+$}.$$
As a consequence, the sequence $(u_{\varepsilon})_{\varepsilon}$ converges strongly to the function $i\vert D_x\vert^{\alpha}u_0$ in $L^2(\mathbb T)$. The expected estimate \eqref{18022024E5} is then obtained by passing to the limit as $\varepsilon\rightarrow0^+$ in \eqref{22022024E2}. Moreover, since $u_{\varepsilon}\in\mathcal N_{b,T-\delta}$ when $0<\varepsilon<\delta$ and that the strong convergence in $L^2(\mathbb T)$ implies almost everywhere pointwize convergence up to extraction, we deduce that $\vert D_x\vert^{\alpha}u_0\in\mathcal N_{b,T-\delta}$ for every $\delta\in(0,T)$. This implies that the function $\vert D_x\vert^{\alpha}u_0$ also belongs to $\mathcal N_{b,T}$. Hence, the vector space $\mathcal N_{b,T}$ is invariant under the action of the linear operator $\vert D_x\vert^{\alpha}$.

Notice that the estimate \eqref{18022024E5} implies that the unit ball of $\mathcal N_{b,T}$ is bounded in $H^{\alpha}(\mathbb T)$, and is therefore relatively compact by Rellich's theorem. This implies that the vector space $\mathcal N_{b,T}$ is finite-dimensional by Riesz' theorem. As a consequence, since the linear operator $\vert D_x\vert^{\alpha}:\mathcal N_{b,T}\rightarrow\mathcal N_{b,T}$ is bounded and in the situation where $\mathcal N_{b,T}\ne\{0\}$, there exists $\lambda\in\mathbb C$ and $u\in\mathcal N_{b,T}\setminus\{0\}$ such that $\vert D_x\vert^{\alpha}u = \lambda u$. The function $u$ not being identically equal to zero on $\mathbb T$, we deduce by passing to the Fourier side that there exists $n\in\mathbb Z$ such that $\lambda = \vert n\vert^{\alpha}$, and that the function $u$ is actually given by
$$u(x) = c_1e^{inx} + c_2e^{-inx},\quad x\in\mathbb T,$$
with $c_1,c_2\in\mathbb C$ such that $c_1c_2\ne0$. Moreover, since $u$ is an eigenfunction of the operator $\vert D_x\vert^{\alpha}$ and $u\in\mathcal N_{b,T}$, we get that the function $u$ vanishes on $\{b>0\}$, a measurable set with positive measure, which is absurd. As a consequence, we have $\mathcal N_{b,T} = \{0\}$ as expected.
\end{proof}

\subsection{Step 4: removing the $H^{-\alpha}$-norm} We are now in position to end the proof of Theorem \ref{thm:fracobs}. We argue once again by contradiction by considering a sequence $(u_{n,0})_n$ in $L^2(\mathbb T)$ such that $\Vert u_{n,0}\Vert_{L^2(\mathbb T)} = 1$ and satisfying
\begin{equation}\label{18022024E3}
	\int_0^T\int_{\mathbb T}b(x)\big\vert(e^{it\vert D_x\vert^{\alpha}}u_{n,0})(x)\big\vert^2\,\mathrm dx\,\mathrm dt \le\frac1n.
\end{equation}
The sequence $(u_{n,0})_n$ being bounded in $L^2(\mathbb T)$, there exists a function $u_0\in L^2(\mathbb T)$ such that, up to extracting, the sequence $(u_{n,0})_n$ converges weakly to 
$u_0$ in $L^2(\mathbb T)$ and strongly in $H^{-\alpha}(\mathbb T)$, by Rellich's theorem. Let us consider the sequence of non-negative functions $(b_j)_j$ in $L^{\infty}(\mathbb T)$ whose elements are given by $b_j = b\mathbbm1_{\{b\le j\}}$ (recall that $b$ is a non-negative function). This sequence $(b_j)_j$ converges to the function $b$ in $L^1(\mathbb T)$ by the dominated convergence theorem.
Notice that the weak convergence, which holds for every $j\geq0$,
$$\forall t\in[0,T],\quad\sqrt{b_j}e^{it\vert D_x\vert^{\alpha}}u_{n,0}\underset{n\rightarrow+\infty}{\rightharpoonup}\sqrt{b_j}e^{it\vert D_x\vert^{\alpha}}u_0\quad\text{in $L^2(\mathbb T)$},$$
is a consequence of the following computations, where $\varphi\in L^2(\mathbb T)$,
\begin{align*}
	\langle\sqrt{b_j}e^{it\vert D_x\vert^{\alpha}}u_{n,0},\varphi\rangle_{L^2(\mathbb T)}
	& = \langle u_{n,0},e^{-it\vert D_x\vert^{\alpha}}\sqrt{b_j}\varphi\rangle_{L^2(\mathbb T)} \\[5pt]
	& \underset{n\rightarrow+\infty}{\longrightarrow}\langle u_0,e^{-it\vert D_x\vert^{\alpha}}\sqrt{b_j}\varphi\rangle_{L^2(\mathbb T)}
	= \langle\sqrt{b_j}e^{it\vert D_x\vert^{\alpha}}u_0,\varphi\rangle_{L^2(\mathbb T)}.
\end{align*}
We therefore deduce from results on weak convergence and Fatou's lemma that
\begin{align}\label{16012025E1}
	\int_0^T\int_{\mathbb T}b_j(x)\big\vert(e^{it\vert D_x\vert^{\alpha}}u_0)(x)\big\vert^2\,\mathrm dx\,\mathrm dt
	& \le \int_0^T\liminf_{n\rightarrow+\infty}\int_{\mathbb T}b_j(x)\big\vert(e^{it\vert D_x\vert^{\alpha}}u_{n,0})(x)\big\vert^2\,\mathrm dx\,\mathrm dt \\[5pt]
	& \le \liminf_{n\rightarrow+\infty}\int_0^T\int_{\mathbb T}b_j(x)\big\vert(e^{it\vert D_x\vert^{\alpha}}u_{n,0})(x)\big\vert^2\,\mathrm dx\,\mathrm dt. \nonumber
\end{align}
Let us split the integral in the right-hand side of the above inequality as follows
\begin{align*}
	\int_0^T\int_{\mathbb T}b_j(x)\big\vert(e^{it\vert D_x\vert^{\alpha}}u_{n,0})(x)\big\vert^2\,\mathrm dx\,\mathrm dt
	 = \int_0^T\int_{\mathbb T}b(x)\big\vert(e^{it\vert D_x\vert^{\alpha}}u_{n,0})(x)\big\vert^2\,\mathrm dx\,\mathrm dt \\[5pt]
	+ \int_0^T\int_{\mathbb T}(b_j-b)(x)\big\vert(e^{it\vert D_x\vert^{\alpha}}u_{n,0})(x)\big\vert^2\,\mathrm dx\,\mathrm dt.
\end{align*}
These two terms can be bounded by using respectively the inequality \eqref{18022024E3} and the $L^{\infty}_xL^2_t$ Strichartz estimate stated in Lemma \ref{11032024L1}, which leads to the following estimate
$$\int_0^T\int_{\mathbb T}b_j(x)\big\vert(e^{it\vert D_x\vert^{\alpha}}u_{n,0})(x)\big\vert^2\,\mathrm dx\,\mathrm dt\lesssim\frac1n + \Vert b-b_j\Vert_{L^1(\mathbb T)},$$
since the functions $u_{n,0}$ are normalized in $L^2(\mathbb T)$. We therefore deduce from \eqref{16012025E1} that
$$\int_0^T\int_{\mathbb T}b_j(x)\big\vert(e^{it\vert D_x\vert^{\alpha}}u_0)(x)\big\vert^2\,\mathrm dx\,\mathrm dt\lesssim\Vert b-b_j\Vert_{L^1(\mathbb T)}.$$
Coming back to the function $b$, this estimate and another use of Lemma \ref{11032024L1} allow to derive that
\begin{align*}
	\int_0^T\int_{\mathbb T}b(x)\big\vert(e^{it\vert D_x\vert^{\alpha}}u_0)(x)\big\vert^2\,\mathrm dx\,\mathrm dt\lesssim\Vert b-b_j\Vert_{L^1(\mathbb T)}\underset{j\rightarrow+\infty}{\longrightarrow}0.
\end{align*}
Keeping the notations of Section \ref{uc}, this proves that $u_0\in\mathcal N_{b,T}$, and therefore $u_0 \equiv 0$ by Lemma \ref{18022024E4}. We finally deduce from Lemma \ref{lem:weakobs} that
$$\underbrace{\Vert u_{n,0}\Vert^2_{L^2(\mathbb T)}}_{=\, 1}\le C_{b,T}\bigg(\underbrace{\int_0^T\int_{\mathbb T}b(x)\big\vert(e^{it\vert D_x\vert^{\alpha}}u_{n,0})(x)\big\vert^2\,\mathrm dx\,\mathrm dt}_{\underset{n\rightarrow+\infty}{\longrightarrow}0} + \underbrace{\Vert u_{n,0}\Vert^2_{H^{-\alpha}(\mathbb T)}}_{\underset{n\rightarrow+\infty}{\longrightarrow}0}\bigg).$$
This is of course absurd, and the proof of Theorem \ref{thm:fracobs} is ended.

\end{document}